\documentclass{amsart}

\usepackage[colorlinks,driverfallback=dvipdfm]{hyperref}
\usepackage{color,graphicx,shortvrb}
\usepackage[latin 1]{inputenc}

\usepackage[active]{srcltx} 

\usepackage{enumerate}
\usepackage{amssymb}
\usepackage{tikz-cd}

\usepackage{amsmath,amssymb,verbatim,amscd,bm,color,yhmath,bm, enumerate}
\usepackage[normalem]{ulem}
\usepackage[driverfallback=dvipdfm]{hyperref}

\usepackage[active]{srcltx} 

\newtheorem{lem}{Lemma}[section]
\newtheorem{thm}[lem]{Theorem}
\newtheorem{prop}[lem]{Proposition}
\newtheorem{cor}[lem]{Corollary}
\newtheorem{remark}[lem]{Remark}

\theoremstyle{definition}

\newtheorem*{ex}{Example}

\numberwithin{equation}{section}

 \newcommand{\cha}{\operatorname{Ch}(A)}

 \DeclareSymbolFont{largesymbol}{OMX}{yhex}{m}{n}
 \DeclareMathAccent{\Widehat}{\mathord}{largesymbol}{"62}

 \DeclareSymbolFont{largesymbol}{OMX}{yhex}{m}{n}

\usepackage{ulem}
\usepackage{marginnote}
\usepackage{geometry}

\begin{document}



\title[Commutative JB$^*$-triples satisfy the complex Mazur--Ulam property]
{Every commutative JB$^*$-triple satisfies the complex  Mazur--Ulam property}

\author[D. Cabezas]{David Cabezas}
\address[D. Cabezas]{
Departamento de An{\'a}lisis Matem{\'a}tico, Facultad de
Ciencias, Universidad de Granada, 18071 Granada, Spain.}
\email{dcabezas@ugr.es}

\author[M. Cueto-Avellaneda]{Mar{\'i}a Cueto-Avellaneda}
\address[M. Cueto-Avellaneda]{
School of Mathematics, Statistics and Actuarial Science, University of Kent, Canterbury, Kent CT2 7NX, UK}
\email{emecueto@gmail.com}

\author[D. Hirota]{Daisuke Hirota}
\address[D. Hirota]{Graduate School of Science and Technology, Niigata University, Niigata 950-2181, Japan}
\email{hirota@m.sc.niigata-u.ac.jp}

\author[T. Miura]{Takeshi Miura}
\address[T. Miura]{Department of Mathematics, Faculty of Science, Niigata University, Niigata 950-2181, Japan}
\email{miura@math.sc.niigata-u.ac.jp}

\author[A.M. Peralta]{Antonio M. Peralta}
\address[A.M. Peralta]{Instituto de Matem{\'a}ticas de la Universidad de Granada (IMAG), Departamento de An{\'a}lisis Matem{\'a}tico, Facultad de
Ciencias, Universidad de Granada, 18071 Granada, Spain.}
\email{aperalta@ugr.es}


\begin{abstract} We prove that every commutative JB$^*$-triple satisfies the complex Mazur--Ulam property. Thanks to the representation theory, we can identify commutative JB$^*$-triples as spaces of complex-valued continuous functions on a principal $\mathbb{T}$-bundle $L$ in the form $$C_0^\mathbb{T}(L):=\{a\in C_0(L):a(\lambda t)=\lambda a(t)\text{ for every } (\lambda,t)\in\mathbb{T}\times L\}.$$ We prove that every surjective isometry from the unit sphere of $C_0^\mathbb{T}(L)$ onto the unit sphere of any complex Banach space admits an extension to a surjective real linear isometry between the spaces. 
\end{abstract}

\maketitle


\section{Introduction}

New recent advances continue improving our understanding of Tingley's problem by enlarging the list of positive solutions, and the range of spaces satisfying the Mazur--Ulam property. As introduced in \cite{ChenDong2011}, a Banach space $X$ satisfies the Mazur--Ulam property if every surjective isometry from its unit sphere onto the unit sphere of any other Banach space admits and extension to a surjective real linear isometry between the spaces. The new advances are struggling to provide new tools and techniques to attack this intriguing open question on the possibility of extending a surjective isometry between the unit spheres of two Banach spaces to a surjective real linear isometry between the spaces. A remarkable outstanding advance has been obtained by T. Banakh in \cite{Banakh21SolutionTingley2dim}, who proved that every 2-dimensional Banach space $X$ satisfies the \emph{Mazur--Ulam property}, that is, every surjective isometry from the unit sphere of $X$ onto the unit sphere of any other Banach space extends to a real linear isometry between the involved Banach spaces. This is in fact the culminating point of deep technical advances (see \cite{Banakh21Smoothtwodim, CabSan19, BanakhCabello21}).\smallskip

The abundance of unitary elements in unital C$^*$-algebras, real von Neumann algebras and JBW$^*$-algebras is a key property to prove that these spaces together with all JBW$^*$-triples satisfy the Mazur--Ulam property (cf. \cite{MoriOza2018, BeCuFerPe2018,KalPe2019}). A prototypical example of non-unital C$^*$-algebra is given by the C$^*$-algebra $K(H),$ of all compact operators on an infinite dimensional complex Hilbert space $H$, or more generally, by a compact C$^*$-algebra (i.e. a $c_0$-sum of $K(H)$-spaces). Compact C$^*$-algebras and weakly compact JBW$^*$-triples are in the list of complex Banach spaces satisfying the Mazur--Ulam property   (see \cite{Pe2020FILOMAT}).\smallskip

Another research line is exploring Tingley's problem in the case of certain function algebras and spaces. The first positive solution to Tingley's problem for a Banach space consisting of analytic functions, apart from Hilbert spaces, was obtained by O. Hatori, S. Oi and R. Shindo Togashi in \cite{HatOiTog}, where it is proved that each  surjective isometry between the unit spheres of two uniform algebras (i.e. closed subalgebras of $C(K)$ containing the constants and separating the points of $K$) can be always extended to a surjective real-linear isometry between the uniform algebras. O. Hatori has gone further by showing that every uniform algebra satisfies the complex Mazur--Ulam property \cite[Theorem 4.5]{Hat2021}. The ``\emph{complex Mazur--Ulam property}'' has been coined by O. Hatori to denote those complex Banach spaces for which every surjective isometry from its unit sphere onto the unit sphere of a complex Banach spaces admits an extension to a real linear mapping between the spaces.\smallskip

The non-unital analogue of uniform algebras is materialized in the notion of uniformly closed function algebra on a locally compact Hausdorff space $L$. We recently showed that each surjective isometry between the unit spheres of two uniformly closed function algebras on locally compact Hausdorff spaces admits an extension to a surjective real linear isometry between these algebras (see \cite{cueto2021exploring}). In the just quoted reference we also proved that Tingley's problem admits a positive solution for any surjective isometry between the unit spheres of two commutative JB$^*$-triples, which are not, in general, subalgebras of the algebra $C_0(L)$ of all complex-valued continuous functions on $L$ vanishing at infinity (see Section \ref{sec: abelian JBstar triples} for the detailed representation of commutative JB$^*$-triples). In this note we shall employ a recent tool developed by O. Hatori in \cite{Hat2021} to infer that a stronger conclusion holds, namely, every commutative JB$^*$-triple satisfies the complex Mazur--Ulam property. Among the consequences we derive that every commutative C$^*$-algebra enjoys the complex Mazur--Ulam property. Commutative unital C$^*$-algebras satisfy the stronger Mazur-Ulam property by a result due to  M. Mori and N. Ozawa \cite{MoriOza2018}. Some other previous achievements in this line were obtained by R. Liu for $C(K, \mathbb{R})$ \cite[Corollary 6]{Liu2007}, and by O. Hatori for closed subspaces of $C_0({L},\mathbb{R})$ separating the points of $L$ and satisfying a topological-geometric hypothesis called condition $(r)$ (see \cite[Theorem 5.4]{Hat2021} in the preprint version).

\section{Preliminaries}\label{sec: preliminaries}

We shall briefly recall some basic terminology to understand the sufficient condition, established by O. Hatori in \cite[Proposition 4.4]{Hat2021}, to guarantee that a Banach space satisfies the complex Mazur--Ulam property.\smallskip

Let $X$ be a real or complex Banach space, and let $X^*,$ $S(X)$ and $\mathcal{B}_{X}$ denote the dual space, the unit sphere and the closed unit ball of $X$, respectively. It is known, thanks to Hahn-Banach theorem or Eidelheit's separation theorem, that maximal convex subsets of $S(X)$ and maximal proper norm closed faces of $\mathcal{B}_{X}$ define the same subsets (cf. \cite[Lemma 3.3]{Tan2016Mn} or \cite[Lemma 3.2]{Tan2016finitevN}). The set of all maximal convex subsets of $S(X),$ equivalently, all maximal proper norm closed faces of $\mathcal{B}_X$, will be denoted by $\mathfrak{F}_X$.  For each $F\in \mathfrak{F}_X$ there exists an extreme point $\varphi$ of the closed unit ball $\mathcal{B}_{X^*}$ such that $F = \varphi^{-1}\{1\}\cap  S(X)$ (cf. \cite[Lemma 3.3]{Tan2016Mn}). The set of all extreme points $\varphi$ of  $\mathcal{B}_{X^*}$ for which $\varphi^{-1}\{1\}\cap S(X)$ is a maximal convex subset of $S(X)$ will be denoted by $\mathcal{Q}_{X}.$ On the latter set we consider the equivalence relation defined by $$\varphi\sim \phi \Leftrightarrow \exists \gamma\in \mathbb{T} = S(\mathbb{K})\hbox{ with }  \varphi^{-1}\{1\}\cap  S(X) = (\gamma \phi)^{-1}\{1\}\cap  S(X),$$
where $\mathbb{K} = \mathbb{R}$ if $X$ is a real Banach space and  $\mathbb{K}= \mathbb{C}$ if $X$ is a complex Banach space. A \emph{set of representatives} for the quotient set $\mathcal{Q}_{X}/\sim$ (or for $\mathfrak{F}_X$) will consist in a subset $\mathbf{P}_{X}$ of $\mathcal{Q}_{X}$ which is formed by precisely one, and only one, element in each equivalent class of $\mathcal{Q}_{X}/\sim$. According to the just commented notation and results, for each $F\in \mathfrak{F}_X$ there exists a unique $\varphi\in \mathbf{P}_{X}$ and $\gamma\in \mathbb{T}$ such that $F = F_{\varphi, \gamma} := (\gamma \varphi)^{-1}\{1\}\cap  S(X)$ (cf. \cite[Lemma 2.5]{Hat2021}). That is, the elements in $\mathfrak{F}_X$ are bijectively labelled by the set $\mathbf{P}_{X}\times \mathbb{T},$ and we can define a bijection $\mathcal{I}_{X} : \mathfrak{F}_X \to \mathbf{P}_{X}\times \mathbb{T}$ labelling the set $\mathfrak{F}_X$.\smallskip

For example, by the classical description of the extreme points of the closed unit ball of the dual of a $C(K)$ space as those functionals of the form $\lambda \delta_t (f) = \lambda f(t)$ ($f\in C(K)$) with $t\in K,$ $\lambda \in \mathbb{T},$ the set $\mathbf{P}_{C(K)} = \{ \delta_t : t\in K\}$ is a set of representatives for $\mathfrak{F}_{C(K)}.$ It is shown in \cite[Example 2.4]{Hat2021} that for a uniform algebra $A$ over a compact Hausdorff space $K$, the set $\{\delta_t : t \in \cha\}$ is a set of representatives for $A$, where $\cha$ denotes the Choquet boundary of $A$.\smallskip

Let $A,B$ be non-empty closed subsets of a metric space $(E,d)$. The usual  Hausdorff distance between $A$ and $B$ is defined by
$$d_H(A,B)=\max\{\sup\limits_{a\in A} d(a,B), \sup_{b\in B} d(b,A)\}.$$ We shall employ this Hausdorff distance to measure distances between elements in $\mathfrak{F}_X.$\smallskip 

According to \cite{Hat2021}, a Banach space $X$ satisfies the \emph{condition of the Hausdorff distance} if the elements in  $\mathfrak{F}_X$ satisfy the following rules:
\begin{equation}\label{eq:condition hausdorff distance}
    d_H(F_{\varphi,\lambda},F_{\varphi',\lambda'})=\begin{cases}
    |\lambda - \gamma\lambda'|,& \hbox{ if } \varphi^{-1} \{1\}\cap S(X)=(\gamma \varphi')^{-1}\{1\}\cap S(X), \\
    2,&  \hbox{ if } \varphi\not\sim \varphi'.
    \end{cases}
\end{equation}
for $\varphi,\varphi'\in \mathcal{Q}_X$ and $\lambda,\lambda'\in\mathbb{T}$. Let $\mathbf{P}_{X}\subset \mathcal{Q}_{X}$ be a set of representatives for $\mathfrak{F}_{X}.$ Having in mind the properties of the mapping $\mathcal{I}_{X}^{-1}$ for $\mathbf{P}_{X},$ the condition of the Hausdorff distance in \eqref{eq:condition hausdorff distance} can be rewritten in the form  \begin{equation}\label{eq:condition hausdorff distance representatives}
    d_H(F_{\varphi,\lambda},F_{\varphi',\lambda'})=\begin{cases}
    |\lambda-\lambda'|, & \hbox{ if }  \varphi=\varphi',\\
    2, & \hbox{ if } \varphi \neq \varphi',
    \end{cases}
\end{equation}
for $\varphi,\varphi'\in \mathbf{P}_X$ and $\lambda,\lambda'\in\mathbb{T}$. Under the light of \cite[Lemma 3.1]{Hat2021} to conclude that a complex Banach space $X$ together with a set of representatives $\mathbf{P}_{X}$ satisfies \eqref{eq:condition hausdorff distance representatives} it suffices to prove that \begin{equation}\label{eq sufficient condition for the Hausdorff distance representatives} F_{\varphi,\lambda}\cap F_{\varphi',\lambda'}\neq\emptyset  \hbox{ for any }\varphi\neq \varphi'\hbox{ in } \mathbf{P}_{X}, \lambda,\lambda'\hbox{ in }\mathbb{T}. \end{equation}

Let us go back to the set $\mathcal{Q}_{X}$ determining the set $\mathfrak{F}_{X}$ of all maximal proper norm closed faces of $\mathcal{B}_{X}$. For $\varphi\in \mathcal{Q}_{X}$ and $\alpha\in\overline{\mathbb{D}} = \mathcal{B}_{\mathbb{K}}$ ($\mathbb{K} = \mathbb{R}$ or $\mathbb{C}$), we set
\begin{equation*}
M_{\varphi,\alpha}=\left\{x\in S(X) : d\left(x,F_{\varphi,\frac{\alpha}{|\alpha|}}\right)\leq 1-|\alpha|,\ d\left(x,F_{\varphi,-\frac{\alpha}{|\alpha|}}\right)\leq 1+|\alpha|\right\},
\end{equation*}
where $\frac{\alpha}{|\alpha|}= 1$ if $\alpha=0$. Note that in this definition the inequalities ``$\leq$'' can be replaced with equalities.
It is known that for each $\varphi$ in a set of representatives $\mathbf{P}_{X}$, the inclusion
$$M_{\varphi,\alpha}\subseteq \{x\in S(X): \varphi(x)=\alpha\}$$
holds for all $\alpha\in \mathbb{D}$ (cf. \cite[Lemma 4.3]{Hat2021}). The equality of these two sets is the second condition required in the following result, which can be applied to guarantee that a Banach space enjoys the Mazur--Ulam property \cite{Hat2021}.

\begin{prop}\label{p Hatori sufficient}{\rm\cite[Proposition 4.4]{Hat2021}} Let $X$ be a complex Banach space. Suppose that $\mathbf{P}_{X}$ is  a set of representatives for $\mathfrak{F}_{X}$. Assume that $X$ satisfies the following two hypotheses:
\begin{enumerate}[$(i)$] 
\item $X$ satisfies the condition of the Hausdorff distance,
\item For each $\varphi$ in $\mathbf{P}_{X}$ and $\alpha \in \mathbb{D}$ we have
$$M_{\varphi,\alpha}=\{x\in S(X): \varphi(x)=\alpha\}.$$
\end{enumerate} Then $X$ satisfies the complex Mazur-Ulam property.
\end{prop}

\section{Mazur-Ulam property for commutative JB*-triples}\label{sec: abelian JBstar triples}

The aim of this section is to show that every abelian JB$^*$-triple satisfies the Mazur--Ulam property. We shall avoid the axiomatic defition of these objects and we shall simply recall their representation as function spaces. It follows from the Gelfand theory for JB$^*$-triples (see \cite[Corollary 1.11]{Ka83}) that each abelian JB$^*$-triple can be identified with the norm closed subspace of $C_0(L)$ defined by 
$$C_0^\mathbb{T}(L):=\{a\in C_0(L):a(\lambda t)=\lambda a(t)\text{ for every } (\lambda,t)\in\mathbb{T}\times L\},$$
where $L$ is a principal $\mathbb{T}$-bundle $L,$ that is, a subset of a Hausdorff locally convex complex space such that $0 \notin L$, $L \cup \{0\}$ is compact, and $\mathbb{T} L = L$ (see also  \cite{cueto2021exploring}).\smallskip

We can state next the main result of the paper.

\begin{thm} \label{thm: jb-triples main thm}
Let $L$ be a principal $\mathbb{T}$-bundle. Then, $C_0^\mathbb{T}(L)$ satisfies the complex Mazur--Ulam property, that is, for each complex Banach space $X$, every surjective isometry $\Delta:S(C_0^\mathbb{T}(L))\rightarrow S(X)$ admits an extension to a surjective real linear isometry $T:C_0^\mathbb{T}(L)\rightarrow X$.
\end{thm}

The proof will be obtained after a series of technical results witnessing that every function space of the form $C_0^\mathbb{T}(L)$ satisfies the hypotheses of Proposition \ref{p Hatori sufficient}. Before dealing with the technical details, we shall present an  interesting corollary.\smallskip

Although for each locally compact space $\tilde{L}$, the Banach space $C_0(\tilde{L})$ is isometrically isomorphic to a $C^{\mathbb{T}}_0 (L)$ space (cf. \cite[Proposition 10]{Ol74}), there exist principal $\mathbb{T}$-bundles $L$ for which the space $C^{\mathbb{T}}_0 (L)$ is not isometrically isomorphic to a $C_0(L)$ space (cf. \cite[Corollary 1.13 and subsequent comments]{Ka83}). So, there exist abelian JB$^*$-triples which are not isometrically isomorphic to commutative C$^*$-algebras. The next corollary is a weaker consequence of our previous theorem. 

\begin{cor}\label{c abelian Cstar algebras satisfy the MUP} Every abelian C$^*$-algebra {\rm(}that is, every $C_0(\tilde{L})$ space{\rm)} satisfies the complex Mazur--Ulam property. 
\end{cor}

Compared with previous results, we observe that as a consequence of the result proved by O. Hatori for uniform algebras in \cite[Theorem 4.5]{Hat2021} every unital abelian C$^*$-algebra satisfies the complex Mazur--Ulam property. Actually, all unital C$^*$-algebras enjoy the Mazur--Ulam property \cite{MoriOza2018}. In the case of real-valued continuous functions, R. Liu proved that for each compact Hausdorff space $K$, $C(K, \mathbb{R})$ satisfies the Mazur--Ulam property (see \cite[Corollary 6]{Liu2007}). Previous attempts were conducted by G. Ding \cite{Ding2003} and X.N. Fang and J.H. Wang \cite {FangWang06}.\smallskip

Let $\tilde{L}$ be a locally compact Hausdorff space. A closed subspace $E$ of $C_0(\tilde{L},\mathbb{K}),$ separates the points of $\tilde{L}$ if for any $t_1\neq t_2$ in $\tilde{L}$ there exists a function $a\in E$ such that $a(t_1) \neq a(t_2)$. Following \cite{Hat2021}, we shall say that $E$ satisfies
the condition $(r)$ if for any $t$ in the Choquet boundary of $E$, each neighborhood $V$ of $t$, and $\varepsilon > 0$ there exists $u \in E$ such that $0 \leq  u \leq 1 = u(t)$ on $\tilde{L}$ and $0 \leq u \leq \varepsilon$ on $\tilde{L}\backslash V$. The proof of Corollary 5.4 in the preprint version of \cite{Hat2021} affirms that each closed subspace $E$ of $C_0(\tilde{L},\mathbb{R})$ separating the points of $\tilde{L}$ and satisfying a stronger assumption than condition $(r)$ has the Mazur--Ulam property. After some private communications with O. Hatori we actually learned that property $(r)$ is enough to conclude that any such closed subspace $E$ satisfies the Mazur--Ulam property. Actually the desired conclusion can be derived from \cite[Theorem 2.4]{BKMW} by just observing that conditon $(r)$ implies that the isometric identification of $E$ in $C_0\left(\overline{Ch(E)}\right)$ is C-rich, and hence a lush space. Corollary 3.9 in \cite{THL} implies that $E$ has the Mazur--Ulam property.\smallskip

We focus now on the main goal of this section.  Henceforth, let $L$ be a principal $\mathbb{T}$-bundle and $L_0\subset L$ a maximal non-overlapping set, that is, $L_0$ is maximal satisfying that for each $t\in L_0$ we have $L_0\cap \mathbb{T} t = \{t\}$ (its existence is guaranteed by Zorn's lemma).\smallskip

Assume that a Banach space $Y$ satisfies the following property: for every extreme point $\varphi\in\partial_e(\mathcal{B}_{Y^*})$, the set $\{\varphi\}$ is a weak$^*$-semi-exposed face of $\mathcal{B}_{Y^*}$. It is clear that each extreme point $\varphi\in\partial_e(\mathcal{B}_{Y^*})$ is determined by the set $\{\varphi\}_{\prime}=\varphi^{-1}(1)\cap S(Y)$. Hence, the equivalence relation $\sim$ defined in Section \ref{sec: preliminaries} (cf. \cite[Definition 2.1]{Hat2021}) can be characterized in the following terms: for $\varphi,\psi\in\partial_e(\mathcal{B}_{Y^*})$, we have 
\begin{equation*}
    \varphi\sim\psi\Longleftrightarrow \varphi=\gamma\psi \hbox{ for some $\gamma\in\mathbb{T}$.}
\end{equation*} Since $Y= C_0^\mathbb{T}(L)$ satisfies the mentioned property, the set $\{\delta_{t}:t\in L_0\}$ is a set of representatives for the relation $\sim$. We know that each maximal proper face $F$ of the closed unit ball of $C_0^\mathbb{T}(L)$ is of the form $$F=F_{\delta_{t_0},\lambda}= F_{t_0,\lambda}:=\{a\in S(C_0^\mathbb{T}(L)):\delta_{t_0}(a)=a(t_0)=\lambda\}$$ for some $(t_0,\lambda)\in L_0\times\mathbb{T}$ (cf. \cite[Lemma 3.5]{cueto2021exploring}).\smallskip

The next remark, which has been borrowed from \cite[Remark 3.4]{cueto2021exploring} states a kind of Urysohn's lemma for the space $C_0^\mathbb{T}(L)$. 

\begin{remark}\label{rmk:T-symmetric Urysohn}{\rm\cite[Remark 3.4]{cueto2021exploring}}{\rm
Suppose $L$ is a principal $\mathbb{T}$-bundle. Let $W$ be a $\mathbb{T}$-invariant open neighbourhood of $t_0$ in $L$ which is contained in a compact $\mathbb{T}$-invariant subset. Then, there exists a function $h\in S(C_0^\mathbb{T}(L))$ satisfying $h(t_0)=1$ and $h(t)=0$ for all $t\in X\backslash W$. }
\end{remark}

We recall next, for later purposes, the following version of \cite[Theorem 2.7]{rud} for $\mathbb{T}$-symmetric compact subsets of principal $\mathbb{T}$-bundles.

\begin{remark}\label{rmk:T-symmetric neighbourhood}{\rm Let $L$ be a principal $\mathbb{T}$-bundle. Let $K$ be a $\mathbb{T}$-symmetric compact subset of $L$ contained in an open $\mathbb{T}$-symmetric set $U\subset L$. Then, there exits an open $\mathbb{T}$-symmetric set $V$ with compact $\mathbb{T}$-symmetric closure such that
$$K\subset V\subset \overline{V}\subset U.$$ The proof follows straightforwardly from \cite[Theorem 2.7]{rud} by just observing that for each open (respectively, closed or compact) subset $\mathcal{O}\subset L,$ the set $\mathbb{T} \mathcal{O}$ is open (respectively, closed or compact).}
\end{remark}

We can now begin with the technical details for our arguments. 

\begin{lem}\label{lem:disjoint T-symmetric neighbourhoods} If $t_1\neq t_2$ in $L_0$, then there exist open $\mathbb{T}$-symmetric subsets $V_1,V_2\subset L$ satisfying:
\begin{enumerate}[$\bullet$]
    \item $\mathbb{T}t_j\subset V_j$ for $j=1,2$;
    \item $\overline{V_j}$ is compact for $j=1,2$;
    \item $V_1\cap V_2=\overline{V_1}\cap\overline{V_2}=\emptyset$.
\end{enumerate}
\end{lem}

\begin{proof}
    Since $L_0$ is non-overlapping, we know that $\mathbb{T}t_1$ and $\mathbb{T}t_2$ are disjoint compact subsets of $L$. Hence
    $$\mathbb{T}t_1\subset L\backslash \mathbb{T}t_2,$$
    where $L\backslash \mathbb{T}t_2$ is $\mathbb{T}$-symmetric and open. By Remark \ref{rmk:T-symmetric neighbourhood}, there exists a $\mathbb{T}$-symmetric open set $V_1\subset L$ with $\mathbb{T}$-symmetric compact closure satisfying
    $$\mathbb{T}t_1\subset V_1\subset\overline{V_1}\subset L\backslash \mathbb{T}t_2.$$
    
    Now, having in mind that $\overline{V_1}$ is $\mathbb{T}$-symmetric and compact with $\overline{V_1}\cap\mathbb{T}t_2=\emptyset$, we deduce that $L\backslash\overline{V_1}$ is an open $\mathbb{T}$-symmetric set containing $\mathbb{T}t_2$. We can find, again via Remark \ref{rmk:T-symmetric neighbourhood}, an open $\mathbb{T}$-symmetric subset $V_2$ with $\mathbb{T}$-symmetric compact closure satisfying
    $$\mathbb{T}t_2\subset V_2\subset\overline{V_2}\subset L\backslash\overline{V_1}.$$
\end{proof}

\begin{cor} \label{cor: non-empty intersection faces} Let $t_1\neq t_2$ in $L_0$ and $\lambda_1,\lambda_2\in\mathbb{T}$. Then, there exist a function $a\in S(C_0^\mathbb{T}(L))$ such that $a(t_j)=\lambda_j$ for $j=1,2$.
\end{cor}

\begin{proof} Lemma \ref{lem:disjoint T-symmetric neighbourhoods} assures the existence of disjoint open $\mathbb{T}$-symmetric neighbourhoods with compact closure $W_1$ and $W_2$ of $t_1$ and $t_2,$ respectively. By Remark \ref{rmk:T-symmetric Urysohn}, there exist functions $a_1,a_2\in S(C_0^\mathbb{T}(L))$ such that $a_j(t_j)=1$ and $a_j\raisebox{-1ex}{$|$}_{L\backslash W_j}\equiv 0$ for $j=1,2$.\smallskip

Since $W_1$ and $W_2$ are disjoint, we have $0\in\{a_1(t),a_2(t)\}$ for each $t\in L$. Hence, $a:=\lambda_1 a_1+\lambda_2 a_2$ lies in $S(C_0^\mathbb{T}(L))$ and satisfies the required conclusion.
\end{proof}

\begin{remark}\label{r abelian JB*-triples satisfy the condition of the Hausdorff distance} {\rm The previous corollary shows that $F_{t_1,\lambda}\cap F_{t_2,\lambda'}\neq\emptyset$ for any $t_1\neq t_2$ in $L_0$ and $\lambda,\lambda'\in\mathbb{T}$. Therefore, the space $C_0^\mathbb{T}(L)$ satisfies \eqref{eq sufficient condition for the Hausdorff distance representatives} and hence the condition of the Hausdorff distance (cf. \eqref{eq:condition hausdorff distance} and \eqref{eq sufficient condition for the Hausdorff distance representatives} or \cite[Lemma 3.1]{Hat2021}).}
\end{remark}


We shall next show that $C^\mathbb{T}_0(L)$ satisfies the second hypothesis in Proposition \ref{p Hatori sufficient}.\smallskip

Along this note the set $ \mathcal{B}_{\mathbb{C}}\backslash\{0\}$ will be regarded as a principal $\mathbb{T}$-bundle, and we shall write  $C^\mathbb{T}_0(\mathcal{B}_{\mathbb{C}})$ for the abelian JB$^*$-triple associated with this principal $\mathbb{T}$-bundle, that is, $$ C^\mathbb{T}_0(\mathcal{B}_{\mathbb{C}}) = \left\{ f\in C_0(\mathcal{B}_{\mathbb{C}}\backslash\{0\}) : f(0) =0, \ f(\lambda t) = \lambda f(t), \ \forall t\in \mathcal{B}_{\mathbb{C}}, \ \lambda\in \mathbb{T}  \right\}.$$

\begin{lem} \label{lem:normalization is in sphere}
For each non-zero $a\in \mathcal{B}_{C^\mathbb{T}_0(L)}$ and each $0<\varepsilon<\|a\|\leq 1$, there exists a function $u_\varepsilon\in S(C^\mathbb{T}_0(L))$ satisfying
\[u_\varepsilon(s)=\frac{a(s)}{|a(s)|}\]
for every $s\in L$ with $|a(s)|\geq\varepsilon$.
\end{lem}

\begin{proof} Let us observe that each function $f\in C^\mathbb{T}_0(\mathcal{B}_{\mathbb{C}})$ is uniquely determined by its values on $[0,1].$ Consider the function $h_\varepsilon\in  C^\mathbb{T}_0(\mathcal{B}_{\mathbb{C}})$ defined on $[0,1]$ by
\begin{equation*}
    h_\varepsilon(t)=\begin{cases}
    0,\hspace{11.5mm} t\leq\varepsilon/2, \\
    \text{affine},\quad \varepsilon/2<t<\varepsilon, \\
    1,\hspace{11.5mm} t\geq\varepsilon.
    \end{cases}
\end{equation*}
Since $\|a\|>\varepsilon$, the function $(h_\varepsilon)_t(a)=h\circ a$ lies in $S(C^\mathbb{T}_0(L))$.\smallskip

We also know that for each $s\in L$ such that $|a(s)|\geq\varepsilon$ we have
\begin{align*}
    u_\varepsilon(s)=h_\varepsilon\big(a(s)\big)=h_\varepsilon\Big(\frac{a(s)}{|a(s)|}|a(s)|\Big)=\frac{a(s)}{|a(s)|}h_\varepsilon(|a(s)|)=\frac{a(s)}{|a(s)|}\cdot 1,
\end{align*}
which concludes the proof.
\end{proof}

The next technical result is the key tool for our purposes.

\begin{lem} \label{lem:element close in face}
Let $t_0\in L$. Suppose that $a\in S(C^\mathbb{T}_0(L))$ satisfies $0<|a(t_0)|<1$.
Then, for each $\varepsilon>0$ with $\varepsilon<\min\{|a(t_0)|,1-|a(t_0)|\}$, there exist $a_\varepsilon, b_\varepsilon\in S(C^\mathbb{T}_0(L))$ such that $\|a-a_\varepsilon\|\leq\frac{\varepsilon}{2},$ and for each $0<r<1$ the function
\begin{equation*}
    c_{r,\varepsilon}:=r a_\varepsilon-(1+r|a(t_0)|) \, b_\varepsilon
\end{equation*}
lies in $F_{t_0,-\frac{a(t_0)}{|a(t_0)|}}$ with   $\|a_\varepsilon-c_{r,\varepsilon}\|\leq 2-r+r|a(t_0)|$.
\end{lem}

\begin{proof}
To simplify the notation we shall write $\alpha=a(t_0)$.
Consider the function $g_\varepsilon\in C^\mathbb{T}_0(\mathcal{B}_\mathbb{C})$ defined on $[0,1]$ by
\begin{equation*}
    g_\varepsilon(t)=\begin{cases}
    t,\hspace{12.2mm} 0\leq t\leq |\alpha|-\varepsilon, \\
    \text{affine},\quad |\alpha|-\varepsilon<t<|\alpha|-\varepsilon/2, \\
    |\alpha|,\hspace{8.8mm} |\alpha|-\varepsilon/2\leq t\leq |\alpha|+\varepsilon/2, \\
    \text{affine},\quad |\alpha|+\varepsilon/2<t<|\alpha|+\varepsilon, \\
    t,\hspace{12.2mm} |\alpha|+\varepsilon\leq t\leq 1.
    \end{cases}
\end{equation*}

By setting $a_\varepsilon=(g_\varepsilon)_t(a)=g_\varepsilon\circ a\in C^\mathbb{T}_0(L)$, we have $\|a_\varepsilon\|=1$ and $a_\varepsilon(t_0)=\alpha$.\smallskip

Let $\iota:\mathcal{B}_\mathbb{C}\hookrightarrow\mathbb{C}$ denote the inclusion mapping $-$note that $\iota\in C^\mathbb{T}_0(\mathcal{B}_\mathbb{C})$.
Clearly, $\|g_\varepsilon-\iota\|=\varepsilon/2$, thus
\[\|a-a_\varepsilon\|=\|(g_\varepsilon)_t(a)-\iota_t(a)\|\leq \varepsilon/2.\] 

Let us find, by the continuity of $a$, an open neighborhood $\mathcal{O}_\varepsilon\subset L$ of $t_0$ satisfying $|a(s)-a(t_0)|<\varepsilon/2,$ for all $s\in\mathcal{O}_\varepsilon$. The set $W_\varepsilon:=\mathbb{T}\mathcal{O}_\varepsilon$ is open and $\mathbb{T}$-symmetric. For each $s\in W_\varepsilon$, we have $s=\lambda_s t_s$ for some $\lambda_s\in\mathbb{T}$ and $t_s\in\mathcal{O}_\varepsilon$, hence
\begin{align*}
    \big||a(s)|-|a(t_0)|\big|=\big||a(\lambda_s t_s)|-|a(t_0)|\big|=\big||a(t_s)|-|a(t_0)|\big|\leq |a(t_s)-a(t_0)|<\varepsilon/2,
\end{align*}
which leads to
\begin{equation}\label{eq:module a near alpha in W}
    |\alpha|-\varepsilon/2=|a(t_0)|-\varepsilon/2<|a(s)|<|a(t_0)|+\varepsilon/2=|\alpha|+\varepsilon/2
\end{equation}
for all $s\in W_\varepsilon$. Consequently,
\begin{equation}\label{eq:a epsilon in W}
   a_\varepsilon(s)=g_\varepsilon\big(a(s)\big)=\frac{a(s)}{|a(s)|}g_\varepsilon\big(|a(s)|\big)=\frac{a(s)}{|a(s)|}|\alpha|
\end{equation}
for all $s\in W_\varepsilon$.\smallskip

Since $W_\varepsilon$ is open, $\mathbb{T}$-symmetric and contains $t_0$,
Remark \ref{rmk:T-symmetric neighbourhood} assures the existence of an open and $\mathbb{T}$-symmetric set $V_\varepsilon$ with compact $\mathbb{T}$-symmetric closure satisfying
\begin{equation}\label{eq:T-symmetric with compact closure}
    \mathbb{T}t_0\subset V_\varepsilon\subset\overline{V_\varepsilon}\subset W_\varepsilon.
    \end{equation}
By Lemma \ref{lem:normalization is in sphere}, there exists $u_\varepsilon\in S(C^\mathbb{T}_0(L))$ satisfying
\begin{equation}\label{eq:normalization is in sphere}
    u_\varepsilon(s)=\frac{a_\varepsilon(s)}{|a_\varepsilon(s)|}, \quad\text{for all $s\in L$ with $|a_\varepsilon(s)|\geq |\alpha|-\varepsilon$}.
\end{equation}
Since $|a_\varepsilon(t_0)|=|\alpha|>|\alpha|-\varepsilon$, we have $u_\varepsilon(t_0)=\frac{\alpha}{|\alpha|}$.\smallskip

Fix, via Remark \ref{rmk:T-symmetric Urysohn}, a function $f_\varepsilon\in S(C^\mathbb{T}_0(L))$ such that $f_\varepsilon(t_0)=1$ and $f_\varepsilon\raisebox{-1ex}{$|$}_{L\backslash V_\varepsilon}\equiv 0$.\smallskip

For each $r\in]0,1[$, let us define the function
\begin{equation*}
    c_{r,\varepsilon}:=r a_\varepsilon-(1+r|\alpha|)\,  |f_\varepsilon| \, u_\varepsilon\in C^\mathbb{T}_0(L).
\end{equation*}

Clearly, we have
\begin{equation}\label{eq:c(t_0)}
   c_{r,\varepsilon}(t_0)=r a_\varepsilon(t_0)-(1+r|\alpha|)|f_\varepsilon(t_0)|u_\varepsilon(t_0)=r\alpha-(1+r|\alpha|)\frac{\alpha}{|\alpha|}=-\frac{\alpha}{|\alpha|}.
\end{equation}
Given $s\in L\backslash V_\varepsilon$, we have $f_\varepsilon(s)=0$. Hence, $c_{r,\varepsilon}(s)=r a_\varepsilon(s)$, so $|c_{r,\varepsilon}(s)|\leq r$.
If $s\in V_\varepsilon\subset W_\varepsilon$, by \eqref{eq:a epsilon in W} we have
\[a_\varepsilon(s)=\frac{a(s)}{|a(s)|}|\alpha|.\]
In particular, $|a_\varepsilon(s)|=|\alpha|>|\alpha|-\varepsilon$, and \eqref{eq:normalization is in sphere} gives 
\[u_\varepsilon(s)=\frac{a_\varepsilon(s)}{|a_\varepsilon(s)|}=\frac{|\alpha|\frac{a(s)}{|a(s)|}}{|\alpha|}=\frac{a(s)}{|a(s)|}.\]
Therefore,
\begin{align*}
    c_{r,\varepsilon}(s)&=r a_\varepsilon(s)-(1+r|\alpha|) \, |f_\varepsilon(s)| \, u_\varepsilon(s)=r|\alpha|\frac{a(s)}{|a(s)|}-(1+r|\alpha|)\frac{a(s)}{|a(s)|} \, |f_\varepsilon(s)| \\
    &=\frac{a(s)}{|a(s)|}\big(r|\alpha|-(1+r|\alpha|)|f_\varepsilon(s)|\big).
\end{align*}
Since $|f_\varepsilon(s)|\in[0,1]$, we have \begin{align*}
     -1-r|\alpha| & \leq - (1+r|\alpha|)|f_\varepsilon(s)|\leq 0, \\
     -1\leq r|\alpha| &- (1+r|\alpha|)|f_\varepsilon(s)|\leq r|\alpha|\leq 1,
\end{align*}
which implies that $\big|r|\alpha|- (1+r|\alpha|) \, |f_\varepsilon(s)|\, \big|\leq 1$ for all $s\in V_\varepsilon$. It follows that $|c_{r,\varepsilon}(s)|\leq 1$, for all $s\in L$.\smallskip

We have shown that $\|c_{r,\varepsilon}\|\leq 1$, and \eqref{eq:c(t_0)} implies that $c_{r,\varepsilon}\in F_{t_0,-\frac{\alpha}{|\alpha|}}\subset S(C^\mathbb{T}_0(L))$.\smallskip

Finally, we compute the distance between $a_\varepsilon$ and $c_{r,\varepsilon}$ to get
\begin{align*}
    \|a_\varepsilon-c_{r,\varepsilon}\|&=\|(1-r) a_\varepsilon + (1+r|\alpha|) \ |f_\varepsilon(s)| u_\varepsilon(s) \| \\
    &\leq 1-r+1+r|\alpha|=2-r+r|\alpha|.
\end{align*}
\end{proof}

We are now in a position to prove that every abelian JB$^*$-triple satisfies the second hypothesis of Proposition \ref{p Hatori sufficient} for $L_0$ as the set of representatives.

\begin{prop} \label{prop: equality M C0T}
	For every $t_0\in L_0$ and $\alpha\in\overline{\mathbb{D}}$ the equality
	$$M_{t_0,\alpha}=M_{\delta_{t_0},\alpha}=\{a\in S(C^\mathbb{T}_0(L)): a(t_0)=\alpha\}$$
	holds.
\end{prop}

\begin{proof}
We only need to show the inclusion $\supseteq$, because the other implication always holds. Take any $a\in S(C^\mathbb{T}_0(L))$ such that $a(t_0)=\alpha$. We shall discuss first the case where $|\alpha|=1$, in such case we have $a\in F_{t_0,\frac{\alpha}{|\alpha|}}=F_{t_0,\alpha}$ and $-a\in F_{t_0,-\frac{\alpha}{|\alpha|}} = F_{t_0,-\alpha}$. Thus, $d(a,F_{t_0,\frac{\alpha}{|\alpha|}})\leq d(a,a)=0=1-|\alpha|$ and $d(a,F_{t_0,-\frac{\alpha}{|\alpha|}})\leq d(a,-a)=2=1+|\alpha|$.\smallskip

We assume next that $\alpha=0$. For each $\varepsilon>0$ we can find an open neighbourhood $U_\varepsilon$ of $t_0$ and an element $a_\varepsilon\in S(C^\mathbb{T}_0(L))$ such that $\|a-a_\varepsilon\|<\varepsilon$ and $a_\varepsilon\raisebox{-1ex}{$|$}_{U_\varepsilon}\equiv 0$. Keeping in mind Remark \ref{rmk:T-symmetric neighbourhood}, we may assume that $U_\varepsilon$ is $\mathbb{T}$-symmetric and has compact $\mathbb{T}$-symmetric closure. Then, by Remark \ref{rmk:T-symmetric Urysohn}, there exists a function $b_\varepsilon\in F_{t_0,1}$ (and thus $-b_\varepsilon\in F_{t_0,-1}$) with $b_\varepsilon\raisebox{-1ex}{$|$}_{L\backslash U_\varepsilon}\equiv 0$. Putting all together, we have
\[d(a,F_{t_0,\pm 1})\leq \|a\mp b_\varepsilon\|\leq \|a-a_\varepsilon\|+\|a_\varepsilon\mp b_\varepsilon\|<\varepsilon+\max\{\|b_\varepsilon\|,\|a_\varepsilon\|\}\leq 1+\varepsilon.\]
Therefore, $d(a,F_{t_0,\pm 1})\leq 1=1\mp|\alpha|$.\smallskip

We finally assume that $0<|\alpha|=|a(t_0)|<1$. For each $\varepsilon>0$, Lemma \cite[Lemma 3.15]{cueto2021exploring}, gives $b_\varepsilon\in F_{t_0,1}$ and $a_\varepsilon\in S(C^\mathbb{T}_0(L))$ such that
	$$c_{r,\varepsilon}=r a_\varepsilon+(1-r|a(t_0)|)\frac{a(t_0)}{|a(t_0)|} b_\varepsilon\in F_{t_0,\frac{\alpha}{|\alpha|}} \quad (\forall 0< r <1 ),$$
	$a_\varepsilon(t_0)=a(t_0)$, and $\|a-a_\varepsilon\|<\varepsilon$.\smallskip

On the one hand, we have
	\begin{align*}
	\|a-c_{r,\varepsilon}\|&=\|a-a_\varepsilon\|+\|a_\varepsilon-c_{r,\varepsilon}\|<\varepsilon+\left\|(1-r)a_\varepsilon-(1-r|\alpha|)\frac{\alpha}{|\alpha|}b_\varepsilon\right\| \\
	&\leq \varepsilon+(1-r)+1-r|\alpha|=2-r-r|\alpha|+\varepsilon
	\end{align*}
	for each $\varepsilon>0$ and $0<r<1$. Letting $\varepsilon \to 0$ and $r\to 1$ we get   $d(a,F_{t_0,\frac{\alpha}{|\alpha|}})\leq 1-|\alpha|$.\smallskip

Similarly, by Lemma \ref{lem:element close in face}, there exist $a'_\varepsilon, b'_\varepsilon\in S(C^\mathbb{T}_0(L))$ such that $\|a-a'_\varepsilon\|\leq \varepsilon/2$, and  for each $0<r<1$, the function
\begin{equation*}
    c'_{r,\varepsilon}:=r  a'_\varepsilon - (1 + r\  |\alpha| ) \ b'_\varepsilon
    \end{equation*}
    lies in $F_{t_0,-\frac{\alpha}{|\alpha|}}$ with $\|a_\varepsilon-c'_{r,\varepsilon}\|\leq 2-r+r|\alpha|$. Thus,
	\begin{align*}
	\|a-c'_{r,\varepsilon}\|=\|a-a'_\varepsilon\|+\|a'_\varepsilon-c'_{r,\varepsilon}\|\leq \varepsilon/2+2-r+r|\alpha|
	\end{align*} for all $\varepsilon>0$ and $0<r<1$. By taking limits in $\varepsilon \to 0$ and $r\to 1$ we arrive at $d(a,F_{t_0,-\frac{\alpha}{|\alpha|}})\leq 1+|\alpha|$. Therefore, $a\in M_{t_0,\alpha}$.
\end{proof}

\begin{proof}[Proof of Theorem \ref{thm: jb-triples main thm}] Remark \ref{r abelian JB*-triples satisfy the condition of the Hausdorff distance} and Proposition \ref{prop: equality M C0T} guarantee that $C^\mathbb{T}_0(L)$ satisfies the hypotheses in Proposition \ref{p Hatori sufficient} \cite[Proposition 4.4]{Hat2021} for the set of representatives given by $L_0$. The just quoted proposition gives the desired conclusion. 
\end{proof}

\smallskip\smallskip

\textbf{Acknowledgements} We would like to thank Professor O. Hatori from Niigata University for sharing with us some of his recent achievements, besides some fruitful discussions and useful comments.\smallskip 

Second author partially supported by EPSRC (UK) project ``Jordan Algebras, Finsler Geometry and Dynamics'' ref. no. EP / R044228 / 1 and by the Spanish Ministry of Science, Innovation and Universities (MICINN) and European Regional Development Fund project no. PGC2018-093332-B-I00, Junta de Andaluc\'{\i}a grants FQM375 and A-FQM-242-UGR18. Fourth author partially supported by JSPS KAKENHI (Japan) Grant Number JP 20K03650. Fifth author partially supported by MCIN / AEI / 10. 13039 / 501100011033 / FEDER ``Una manera de hacer Europa'' project no. PGC2018-093332-B-I00, Junta de Andaluc\'{\i}a grants FQM375, A-FQM-242-UGR18 and PY20$\underline{\ }$ 00255, and by the IMAG--Mar{\'i}a de Maeztu grant CEX2020-001105-M / AEI / 10.13039 / 501100011033.


\end{document}